\documentclass[a4paper,english]{amsart}
\usepackage{mathptmx}
\usepackage[scaled=0.9]{helvet}
\usepackage{courier}
\usepackage[T1]{fontenc}
\usepackage[latin1]{inputenc}
\usepackage{amsthm}
\usepackage{amssymb}

\makeatletter

\pdfpageheight\paperheight
\pdfpagewidth\paperwidth

\numberwithin{equation}{section}
\numberwithin{figure}{section}
\theoremstyle{plain}
\newtheorem{thm}{\protect\theoremname}[section]
  \theoremstyle{plain}
  \newtheorem{prop}[thm]{\protect\propositionname}
  \theoremstyle{remark}
  \newtheorem*{rem*}{\protect\remarkname}
  \theoremstyle{plain}
  \newtheorem{cor}[thm]{\protect\corollaryname}
  \theoremstyle{plain}
  \newtheorem*{conjecture*}{\protect\conjecturename}
  \theoremstyle{remark}
  \newtheorem*{acknowledgement*}{\protect\acknowledgementname}
  \theoremstyle{plain}
  \newtheorem{fact}[thm]{\protect\factname}
  \theoremstyle{plain}
  \newtheorem{lem}[thm]{\protect\lemmaname}
  \theoremstyle{definition}
  \newtheorem{defn}[thm]{\protect\definitionname}

\usepackage{amssymb,amsmath}
\usepackage{bbm}
\usepackage{amscd}
\usepackage{paralist}

\newcommand{\e}{\mathrm e}

\renewcommand{\phi}{\varphi}

\renewcommand{\tilde}{\widetilde}

\DeclareMathOperator{\id}{id}
\DeclareMathOperator{\supp}{supp}
\DeclareMathOperator{\card}{card}
\DeclareMathOperator{\growth}{growth}

\def\R{\mathbb{R}}

\def\N{\mathbb{N}}

\SetSymbolFont{operators}{bold}{OT1}{cmr}{bx}{n}
\SetSymbolFont{letters}{bold}{OML}{cmm}{b}{it}
\SetSymbolFont{symbols}{bold}{OMS}{cmsy}{b}{n}

\usepackage{babel}

\numberwithin{equation}{section} 
\numberwithin{figure}{section} 

\usepackage{hyperref}
\@ifundefined{textmu}
 {\usepackage{textcomp}}{}

\makeatother

\usepackage{babel}
  \providecommand{\acknowledgementname}{Acknowledgement}
  \providecommand{\conjecturename}{Conjecture}
  \providecommand{\corollaryname}{Corollary}
  \providecommand{\definitionname}{Definition}
  \providecommand{\factname}{Fact}
  \providecommand{\lemmaname}{Lemma}
  \providecommand{\propositionname}{Proposition}
  \providecommand{\remarkname}{Remark}
\providecommand{\theoremname}{Theorem}

\begin{document}


\title{Growth and cogrowth of normal subgroups of a free group}

\author{Johannes Jaerisch}

\author{Katsuhiko Matsuzaki}
\begin{abstract}
We give a sufficient condition for a sequence of normal subgroups
of a free group to have the property that both, their growths tend
to the upper bound and their cogrowths tend to the lower bound. The
condition is represented by planarity of the quotient graphs of the
tree.
\end{abstract}

\maketitle

\section{Introduction and statements of result}

We denote by $F_{n}$ the free group of rank $n\ge2$ with a free
set of generators $S$. Let $T_{n}$ denote the Cayley graph of $F_{n}$
with respect to $S$. We equip $T_{n}$ with the word metric $d$.
Let $G<F_{n}$ be a subgroup of $F_{n}$. The Poincar\'{e} exponent
$\delta(G)$ of $G$ is given by 
\[
\delta(G):=\limsup_{R\rightarrow\infty}\frac{1}{R}\log N_{G}(R),\quad\mbox{where }N_{G}(R):=\card\left\{ g\in G\mid d(\id,g)\le R\right\} .
\]
It is well known that $\delta(G)$ is given by the exponent of convergence
of the Poincar\'{e} series 
\[
\delta(G)=\inf\biggl\{ s>0\big|\sum_{g\in G}\e^{-sd(\id,g)}<\infty\biggr\}.
\]
For a normal subgroup $\left\{ 1\right\} \neq N\triangleleft F_{n}$
the ratio 
\[
\eta(F_{n}/N):=\frac{\delta(N)}{\delta(F_{n})}=\frac{\delta(N)}{\log(2n-1)}
\]
is known as the cogrowth of the group presentation $F_{n}/N$, which
was introduced by Grigorchuk (\cite{MR599539}). We have $\eta(F_{n}/N)\le1$
and by a well-known result of Grigorchuk (\cite{MR599539}) we have
that $\eta(F_{n}/N)=1$ if and only if $F_{n}/N$ is amenable. This
criterion is deduced by combining Grigorchuk's cogrowth formula (\cite{MR599539},
see (\ref{eq:cogrowth-EPS}) below) and results of Kesten on random
walks on countable groups (\cite{MR0109367,MR0112053}, see also Cohen
\cite{MR678175}). 

In this paper we focus on Grigorchuk's lower bound of the cogrowth
$\eta(F_{n}/N)>1/2$ (\cite{MR599539}). We have the following more
general results due to Roblin (\cite{MR2166367}). We say that $G<F_{n}$
is of divergence type if $\sum_{g\in G}\e^{-\delta(G)d(\id,g)}=\infty$. 
\begin{prop}
[\cite{MR2166367}] \label{prop:lowerbounds} Let $G<F_{n}$ and $\left\{ 1\right\} \neq N\triangleleft G$.
Then we have the following:
\begin{enumerate}
\item $\delta(N)\ge\delta(G)/2$. 
\item If $G$ is of divergence type, then $\delta(N)>\delta(G)/2$. 
\end{enumerate}
\end{prop}
\begin{rem*}
 For Kleinian groups the first assertion was proved by Falk and Stratmann
in \cite{MR2097162}. An alternative proof of (1) can be given by
following the arguments in \cite{Jaerisch12a}.  An ergodic-theoretic
proof of $\delta(N)>\delta(F_{n})/2$ was recently given in \cite{Jaerisch12b}.
Results in a more general setting have been also obtained in \cite{MYJ2015}.
\end{rem*}
Another important notion we consider is the growth of graphs. By a
graph we mean an unoriented graph with countable vertex set, bounded
vertex degree and loops as well as multiple edges allowed. For a connected
graph $\Gamma$ with path metric $d$ and for some/any $\gamma_{0}\in\Gamma$,
 the growth is given by 
\[
\growth(\Gamma)=\limsup_{n\rightarrow\infty}(\card\{\gamma\in\Gamma\mid d(\gamma,\gamma_{0})\le n\})^{1/n}.
\]
Let $H<F_{n}$ be a subgroup. We consider the action of $H$ on $T_{n}$
by left-multiplication. We denote by $\Gamma_{H}:=T_{n}/H$ the quotient
graph with vertex set 
\[
\mathcal{V}(\Gamma_{H}):=\left\{ Hg\mid g\in F_{n}\right\} ,
\]
where $Hx$ and $Hy\in\mathcal{V}(\Gamma_{H})$ are connected by an
edge if and only if there exists $s\in S$ such that $Hxs=Hy$ or
$Hys=Hx$. 

The following result on growth tightness of $F_{n}$ was proved in
\cite{MR1436550}. 
\begin{prop}
[\cite{MR1436550}] Let $N$ be a normal subgroup of $F_{n}$ with
infinite index. Then 
\[
\growth(\Gamma_{N})<\growth(T_{n})=2n-1.
\]

\end{prop}
A generalization of the growth tightness to hyperbolic groups was
obtained by Arzhantseva and Lysenok in \cite{growthtight2002}. 

The main result of this paper is the following, which gives a sufficient
condition on a sequence of normal subgroups of $F_{n}$ under which
both the growth and the cogrowth converge to their bounds simultaneously.
Recall that a graph is called planar if there exists an embedding
in the sphere. The condition for a finite graph to be planar is known
as Wagner's theorem (\cite{planer1937}). 
\begin{thm}
\label{thm:main}Let $N_{k}\triangleleft F_{n}$, $k\in\N$, be a
sequence of normal subgroups such that $\Gamma_{N_{k}}$ is planar.
Let $\ell_{k}:=\min\left\{ \ell\in\N\mid\exists h\in N_{k},\,\,d(\id,h)=\ell\right\} $.
If $\ell_{k}\rightarrow\infty$, as $k\rightarrow\infty$, then we
have 
\[
\lim_{k\rightarrow\infty}\delta(N_{k})=\frac{1}{2}\delta(F_{n})\quad\mbox{and}\quad\lim_{k\rightarrow\infty}\growth(\Gamma_{N_{k}})=\growth(T_{n}).
\]

\end{thm}
We see that a particular sequence of normal subgroups satisfies the
above condition. For cogrowth this was proved by Grigorchuk (\cite{MR599539}),
and for growth this follows from a result by Shukhov (\cite{shukhov1999}).
For $g_{1},\dots,g_{s}\in F_{n}$ we denote by $\langle\langle g_{1},\dots,g_{s}\rangle\rangle$
the normal closure of $\left\{ g_{1},\dots,g_{s}\right\} $ in $F_{n}$. 
\begin{cor}
Let $F_{n}=\left\langle g_{1},\dots,g_{n}\right\rangle $ and let
$N_{k}:=\langle\langle g_{1}^{k_{1}},g_{2}^{k_{2}},\dots,g_{s}^{k_{s}}\rangle\rangle\vartriangleleft F_{n}$,
$s\in\N$, and put $k=\min\left\{ k_{i}\mid1\le i\le s\right\} $.
Then we have 
\[
\lim_{k\rightarrow\infty}\delta(N_{k})=\frac{1}{2}\delta(F_{n})\quad\mbox{and}\quad\lim_{k\rightarrow\infty}\growth(\Gamma_{N_{k}})=\growth(T_{n}).
\]
\end{cor}
\begin{proof}
We have to verify that $\Gamma_{N}$ is planar for $N=N_{k}$. Then
the corollary follows from Theorem \ref{thm:main}. To prove this consider the graph $\Gamma_{F_n}=T_n/F_n$, which consists of one vertex and $n$ edges that are the loops based at the vertex. We embed this graph into $(n+1)$-punctured sphere $S$ so that each loop of $\Gamma_{F_n}$ wraps around a different puncture and hence the inclusion map induces the isomorphism  $\theta:\pi_1(\Gamma_{F_n})=F_n \to \pi_1(S)$ between their fundamental groups.
By \cite[Proposition X.A.3]{MR959135},  if $w_1,\ldots, w_n \in \pi_1(S)$ correspond to mutually disjoint non-trivial simple closed curves in $S$, then the normal closure $H=\langle \langle w_1^{k_1},\ldots, w_n^{k_n} \rangle \rangle$ for any $k_1,\ldots, k_n \in \mathbb N$ defines the normal covering surface $\widetilde S$ that is planar, namely, all simple closed curves are dividing. Note that any planar surface can be embedded into the sphere.
Set $N=\theta^{-1}(H)<F_n$. By the covering projection $p:\widetilde S \to S$, we lift the embedded graph $\Gamma_{F_n} \subset S$ to $\widetilde S$. Then the lifted graph is $\Gamma_N$. Since $\widetilde S$ is planar, we see that $\Gamma_N$ is a planar graph.
\end{proof}
For the more general case of $G<F_{n}$ we have the following immediate
consequence of our results.
\begin{cor}
Let $G=\left\langle g_{i}:i\in I\right\rangle $ be a subgroup of
$F_{n}$ such that $\delta(G)=\delta(F_{n})$. Then there exists a
sequence of normal subgroups $N_{k}\triangleleft G$, $k\in\N$, such
that 
\[
\lim_{k\rightarrow\infty}\delta(N_{k})=\frac{1}{2}\delta(G).
\]
 \end{cor}
\begin{proof}
Put $\tilde{N}_{k}:=\langle\langle g_{1}^{k}\rangle\rangle\triangleleft F_{n}.$
It is easy to see that $\sup_{k}d(\id,g_{1}^{k})=\infty$. Hence,
by passing to a subsequence, we may assume that $\ell_{k}\rightarrow\infty$,
as $k\rightarrow\infty$. By Theorem \ref{thm:main} we have $\lim_{k\rightarrow\infty}\delta(\tilde{N}_{k})=\delta(F_{n})/2$.
Put $N_{k}:=\tilde{N}_{k}\cap G$. Since $N_{k}\triangleleft G$ we
have $\delta(G)/2\le\delta(N_{k})\le\delta(\tilde{N}_{k})$ by Proposition
\ref{prop:lowerbounds}. The corollary follows because $\lim_{k\rightarrow\infty}\delta(\tilde{N}_{k})=\delta(F_{n})/2=\delta(G)/2$. 
\end{proof}
To prove our main result, we make use of the concept of isoperimetric
inequalities to estimate the bottom of the spectrum of the discrete
Laplacian on graphs. This concept also allows us to give new proofs
for known results on the growth and cogrowth of quotient graphs of
the tree. In Section 2 we introduce the necessary preliminaries on
the discrete Laplacian and isoperimetric inequalities. The proof of
Theorem \ref{thm:main} is given in Section 3. Finally, in Section
4 we derive a relation between the growth and the cogrowth in Proposition
\ref{prop:growth cogrowth relation}, which motivates the following
conjecture.
\begin{conjecture*}
For every non-trivial $N\triangleleft F_{n}$ we have
\[
\delta(N)+\frac{1}{2}\log\left(\growth(T_{n}/N)\right)\ge\delta(F_{n}).
\]

\end{conjecture*}
If this conjecture is true, then $\lim_{k \to \infty} \delta(N_k) = \delta(F_n)/2$ implies $\lim_{k \to \infty} {\rm growth}(T_n/N_k) = {\rm growth}(T_n)$.

\begin{acknowledgement*}
We would like to thank Wenyuan Yang for introducing the growth tightness
of (relatively) hyperbolic groups to us. 
\end{acknowledgement*}

\section{Preliminaries}

\subsection{Discrete Laplacian}

Let $n\in\N$ and let $\Gamma$ be a $(2n)$-regular graph with vertex
set $\mathcal{V}(\Gamma)$. The transition operator of the simple
random walk on $\Gamma$ is for $f:\mathcal{V}(\Gamma)\rightarrow\R$
given by 
\[
Pf(x):=\frac{1}{2n}\sum_{y\sim x}f(y),\quad x\in\mathcal{V}(\Gamma),
\]
where the sum is taken over all edges connecting $x$ and $y$. The
discrete Laplacian $\triangle$ on $\Gamma$ is given by $\triangle f:=f-Pf$.
Denote by $\lambda_{0}(\Gamma)$ the bottom of the spectrum of $\triangle$
given by 
\[
\lambda_{0}(\Gamma):=\inf\left\{ \lambda\in\R\mid\exists f\in\ell^{2}(\mathcal{V}(\Gamma))\mbox{ s.t. }\triangle(f)=\lambda f\right\} .
\]
The following two facts are well known.
\begin{fact}
The bottom of the spectrum is given by 
\[
\lambda_{0}(\Gamma)=\inf_{f:\mathcal{V}(\Gamma)\rightarrow\R,\card(\supp(f))<\infty}\frac{1}{2n}\frac{\sum_{x\sim y}\left|f(x)-f(y)\right|^{2}}{\sum_{x}\left|f(x)\right|^{2}}.
\]

\end{fact}

\begin{fact}
Let $\lambda\in\R$. Then we have $\lambda\le\lambda_{0}(\Gamma)$
if and only if there exists $f:\mathcal{V}(\Gamma)\rightarrow\R_{>0}$
such that $\triangle f=\lambda f$. 
\end{fact}
In order to explain the relation between the bottom of the spectrum
and the Poincar\'{e} exponent, we will first state Girgorchuk's cogrowth
formula. Denote by $\rho(F_{n}/N)$ the spectral radius of the transition
operator $P:\ell^{2}(\mathcal{V}(\Gamma_{N}))\rightarrow\ell^{2}(\mathcal{V}(\Gamma_{N}))$
of the simple random walk on the quotient graph $\Gamma_{N}:=T_{n}/N$. 
\begin{thm}
[Grigorchuk's cogrowth formula] For every $\left\{ 1\right\} \neq N\triangleleft F_{n}$
we have 
\begin{equation}
\rho(F_{n}/N)=\frac{\sqrt{2n-1}}{2n}\left(\frac{\sqrt{2n-1}}{\e^{\delta(N)}}+\frac{\e^{\delta(N)}}{\sqrt{2n-1}}\right).\label{eq:cogrowth-EPS}
\end{equation}
    \end{thm}
\begin{rem*}
In \cite{MR1436550} the cogrowth formula is stated for arbitrary
subgroups $H<F_{n}$. A proof of this formula can be given by using
the Patterson-Sullivan theory.
\end{rem*}
The relation between the bottom of the spectrum of the Laplacian and
the Poincar\'{e} exponent is stated in the following proposition.
The analogue result for Kleinian groups is known as the Theorem of
Elstrodt, Patterson and Sullivan (\cite{MR882827}). 
\begin{prop}
\label{prop:EPS}For every $\{1\}\neq N\triangleleft F_{n}$ we have
\[
\lambda_{0}(\Gamma_{N})=\frac{1}{2n}(2n-1-\e^{\delta(N)})(1-\e^{-\delta(N)}).
\]
 \end{prop}
\begin{proof}
First observe that $\lambda_{0}(\Gamma_{N})=1-\rho(F_{n}/N)$, where
we used the fact that $\rho(F_{n}/N)$ is contained in the spectrum
of the transition operator of the simple random walk on $F_{n}/N$
(\cite{MR943998}, see also \cite[Theorem 4.4]{MR986363}). The proposition
now follows from (\ref{eq:cogrowth-EPS}). More precisely, we have
that  
\begin{align*}
\lambda_{0}(\Gamma_{N}) & =1-\frac{\sqrt{2n-1}}{2n}\left(\frac{\sqrt{2n-1}}{\e^{\delta(N)}}+\frac{\e^{\delta(N)}}{\sqrt{2n-1}}\right)=1-\left(\frac{2n-1}{2n}\e^{-\delta(N)}+\frac{1}{2n}\e^{\delta(N)}\right)\\
 & =\frac{1}{2n}\left(2n-(2n-1)\e^{-\delta(N)}-\e^{\delta(N)}\right)=\frac{1}{2n}\left(2n-1-\e^{\delta(N)}\right)\left(1-\e^{-\delta(N)}\right).
\end{align*}

\end{proof}

\subsection{Isoperimetric constant}

The isoperimetric constant of a $(2n)$-regular graph $\Gamma$ is
given by 
\[
i(\Gamma):=\inf_{A\subset\mathcal{V}(\Gamma),\card(A)<\infty}\frac{1}{2n}\frac{\card(\partial A)}{\card(A)},
\]
where $\partial A$ denotes the set of edges $e$ such that $e$ connects
$x,y$ with $x\in A$ and $y\in\mathcal{V}(\Gamma)\setminus A$. It
is well known that 
\[
i(T_{n})=(n-1)/n\quad\mbox{and}\quad\lambda_{0}(T_{n})=1-(\sqrt{2n-1})/n.
\]

The following analogue of the well-known Cheeger inequality was proved
by Mohar.
\begin{prop}
[\cite{MR943998}, Theorem 2.1] \label{prop:cheeger inequality}We
have  
\[
i(\Gamma)\le\sqrt{1-(1-\lambda_{0}(\Gamma))^{2}}.
\]

\end{prop}
The following relation between $\lambda_{0}$ and the growth is due
to Mohar (\cite{MR943998}).
\begin{lem}
[\cite{MR943998}, Theorem 4.1] \label{lem:growth-lowerbound}We
have 
\[
\growth(\Gamma)\ge\frac{1+i(\Gamma)}{1-i(\Gamma)}.
\]
That is, we have 
\[
i(\Gamma)\le\frac{\growth(\Gamma)-1}{\growth(\Gamma)+1}.
\]

\end{lem}

\section{Proof of the Main Result}

In order to obtain estimates on the isoperimetric constant we show
that, for a subgroup $H<F_{n}$, it suffices to consider all the finite
core subgraphs of $\Gamma_{H}$ in the definition of $i(\Gamma_{H})$. 
\begin{defn}
Let $\Gamma\subset\Gamma_{H}$ be a finite subgraph. The minimal subgraph
$C\subset\Gamma$ such that the inclusion $\iota:C\rightarrow\Gamma$
is a homotopy equivalence is called the core of $\Gamma$. \end{defn}
\begin{lem}
\label{lem:isoperimetric-reduction}Suppose that $H\neq\{1\}$. Then
we have\emph{
\[
i(\Gamma_{H})=\inf\frac{1}{2n}\frac{\card(\partial C_{\Gamma})}{\card(C_{\Gamma})},
\]
}where the infimum is taken over all finite connected subgraphs $\Gamma\subset\Gamma_{H}$
and $C_{\Gamma}$ denotes the core of $\Gamma$.\end{lem}
\begin{proof}
To prove that we can restrict to connected subgraphs, we use the fact
that, if $(a/b)\le(c/d)$ then $(a/b)\le\left(a+c\right)/(b+d)$.
To prove the reduction to finite core subgraphs, let $\Gamma\subset\Gamma_{H}$
be a finite connected subgraph and consider the core of $\Gamma$.
Note that the core is obtained from $\Gamma$ by successively removing
vertices $v\in\Gamma$ for which there exists only one $v'\in\Gamma$
such that $v\sim v'$. Hence, it suffices to show that 
\begin{equation}
\frac{1}{2n}\frac{\card(\partial\Gamma)-2(n-1)}{\card(\Gamma)-1}\le\frac{1}{2n}\frac{\card(\partial\Gamma)}{\card(\Gamma)}.\label{eq:reduction}
\end{equation}
We may assume that $\card(\Gamma)\ge2$. If 
\begin{equation}
\frac{1}{2n}\frac{\card(\partial\Gamma)}{\card(\Gamma)}\le\frac{n-1}{n},\label{eq:isoperi-estimate}
\end{equation}
then (\ref{eq:reduction}) follows from the fact that $(a-c)/(b-d)\le a/b$,
whenever $a/b\le c/d$ and $b>d$. If $i(\Gamma_{H})<(n-1)/n$ then
we may without loss of generality assume that (\ref{eq:isoperi-estimate})
holds. If $i(\Gamma_{H})=(n-1)/n$ then the lemma holds, because the
infimum is attained if we consider a single cycle, which defines a
core subgraph. 
\end{proof}
We denote by $\chi$ the Euler characteristic of a topological space.
\begin{lem}
\label{lem:eulerformula for coregraphs}If $C$ is a connected core
subgraph, then 
\[
\card(\partial C)=(2n-2)\card(C)+2\chi(C).
\]
\end{lem}
\begin{proof}
First observe that the formula holds when $C$ is a single loop, that
is $\chi(C)=0$. The general case follows by induction on the Euler
characteristic, because if we remove a cycle of edges or an edge path,
2 boundary elements appear and the Euler characteristic increases
by 1.\end{proof}
\begin{defn}
The injectivity radius of a connected graph $\Gamma$ is given by
\[
\ell(\Gamma):=\inf_{x\in\mathcal{V}(\Gamma)}\left\{ \ell_{x}(\Gamma)\right\} ,
\]
where we have set 
\[
\ell_{x}(\Gamma):=\frac{1}{2}\min\left\{ \mbox{length}(\gamma)\mid\gamma\mbox{ is non-backtracking edge path from }x\mbox{ to }x\right\} .
\]
Note that if $\Gamma=F_n$ then $\ell(\Gamma)=\infty$, and if a graph $C$ consists of a single vertex and no edge then  $\ell(C)=\infty$. Moreover, if $C$ is a subgraph of  $\Gamma$ then $\ell(C) \geq \ell(\Gamma)$.\end{defn}
\begin{prop}
\label{prop:subgraph lower bound via injectivity radius}Let $C\subset\Gamma_{H}$
be a core subgraph. Suppose that $\ell(C)<\infty$ and that $C$ is
planar. Then we have 
\[
\card(C)\ge\left(-\chi(C)+2\right)\cdot\left(\ell(C)-1\right).
\]
\end{prop}
\begin{proof}
Since $C$ is planar, we can consider $C$ as a subspace of the sphere
$S^{2}$. Hence, 
\[
\chi(C)+\card\left\{ \mbox{faces of }C\right\} =\chi(S^{2})=2,
\]
giving that 
\begin{equation}
\card\left\{ \mbox{faces of }C\right\} =-\chi(C)+2\label{eq:faces via euler}
\end{equation}
Since every edge of $C$ is between two faces and each face is bounded
by at least $2\ell(C)$ edges, we obtain 
\begin{align*}
2\card\left\{ \mbox{edges of }C\right\}  & \ge2\ell(C)\card\left\{ \mbox{faces of }C\right\} .
\end{align*}
Combining with (\ref{eq:faces via euler}) yields
\[
\card\left\{ \mbox{edges of }C\right\} \ge\ell(C)\left(-\chi(C)+2\right).
\]
Finally, we deduce for the number of vertices that 
\[
\card\left\{ \mbox{vertices of }C\right\} =\card\left\{ \mbox{edges of }C\right\} +\chi(C)\ge\left(\ell(C)-1\right)\left(-\chi(C)+2\right).
\]
\end{proof}
\begin{prop}
\label{prop:injectivity tending to infinity}Let $(H_{k})$ be a sequence
of non-trivial subgroups of $F_{n}$. 
\begin{enumerate}
\item Suppose that $\Gamma_{H_{k}}$ is planar for each $k\in\N$. If $\ell(\Gamma_{H_{k}})\rightarrow\infty$,
as $k\rightarrow\infty$,  then $\lim_{k\rightarrow\infty}i(\Gamma_{H_{k}})=i(T_{n})$. 
\item If $\lim_{k\rightarrow\infty}i(\Gamma_{H_{k}})=i(T_{n})$ and $H_{k}\triangleleft F_{n}$,
then 
\[
\lim_{k\rightarrow\infty}\delta(H_{k})=\frac{1}{2}\log(2n-1)=\delta(F_{n})/2.
\]
 
\item If $\lim_{k\rightarrow\infty}i(\Gamma_{H_{k}})=i(T_{n})$, then $\lim_{k\rightarrow\infty}\growth(\Gamma_{H_{k}})=2n-1=\growth(T_{n})$.
\end{enumerate}
\end{prop}
\begin{proof}
We first prove (1). By Lemma \ref{lem:isoperimetric-reduction}, Lemma
\ref{lem:eulerformula for coregraphs} and Proposition \ref{prop:subgraph lower bound via injectivity radius}
we have 
\begin{align*}
i(\Gamma_{H}) & =\inf\frac{1}{2n}\frac{\card(\partial C_{\Gamma})}{\card(C_{\Gamma})}=\inf\frac{1}{2n}\frac{(2n-2)\card(C_{\Gamma})+2\chi(C_{\Gamma})}{\card(C_{\Gamma})}\\
 & =\frac{n-1}{n}+\inf\frac{1}{n}\frac{\chi(C_{\Gamma})}{\card(C_{\Gamma})}\ge\frac{n-1}{n}+\inf\frac{1}{n}\frac{\chi(C_{\Gamma})}{\left(-\chi(C_{\Gamma})+2\right)\cdot\left(\ell(C_{\Gamma})-1\right)}\\
 & >\frac{n-1}{n}-\frac{1}{n\left(\ell(\Gamma_{H})-1\right)},
\end{align*}
where the infimum is taken over all finite connected subgraphs $\Gamma\subset\Gamma_{H}$
and $C_{\Gamma}$ denotes the core of $\Gamma$. Since $i(\Gamma_{H})\le i(T_{n})=(n-1)/n$,
the assertion in (1) follows. 

To prove (2) observe that by Proposition \ref{prop:cheeger inequality}
we have $\lambda_{0}(\Gamma_{H_{k}})\ge1-\sqrt{1-i(\Gamma_{H_{k}})^{2}}$.
Consequently, if $\lim_{k\rightarrow\infty}i(\Gamma_{H_{k}})=i(T_{n})=(n-1)/n$,
then $\liminf_{k\rightarrow\infty}\lambda_{0}(\Gamma_{H_{k}})\ge1-\sqrt{2n-1}\big/n$.
By Proposition \ref{prop:EPS} we conclude that $\limsup_{k\rightarrow\infty}\delta(H_{k})\le\log\sqrt{2n-1}$.
Since $H_{k}\triangleleft F_{n}$ we have $\delta(H_{k})\ge\log\sqrt{2n-1}$
by Proposition \ref{prop:lowerbounds}, which finishes the proof of
(2). 

Finally we turn to the proof of (3). By Proposition \ref{lem:growth-lowerbound}
we have as $k\rightarrow\infty$, 
\[
\growth(\Gamma_{H_{k}})\ge\frac{1+i(\Gamma_{H_{k}})}{1-i(\Gamma_{H_{k}})}\rightarrow\frac{1+(n-1)/n}{1-(n-1)/n}=2n-1=\growth(T_{n}).
\]

\end{proof}

\begin{proof}
[Proof of Theorem \ref{thm:main}] The first assertion follows from
Proposition \ref{prop:injectivity tending to infinity} (1) and (2).
The second assertion follows from Proposition \ref{prop:injectivity tending to infinity}
(1) and (3).
\end{proof}

\section{A relation between Growth and Cogrowth}

We prove a relation between growth and cogrowth for $N\triangleleft F_{n}$. 
\begin{prop}
\label{prop:growth cogrowth relation}For every non-trivial $N\triangleleft F_{n}$
we have 
\[
\delta(N)+\frac{1}{2}\log\left(\growth(T_{n}/N)\right)+\log(2)>\delta(F_{n}).
\]
\end{prop}
\begin{proof}
For ease of notation we write $\delta=\delta(N)$, $\lambda_{0}=\lambda_{0}(T_{n}/N)$
and $\kappa=\log\left(\growth\left(T_{n}/N\right)\right).$ It follows
from Proposition \ref{prop:EPS} that 
\[
\e^{\delta}=n(1-\lambda_{0})+\sqrt{n^{2}(1-\lambda_{0})^{2}-(2n-1)}.
\]
By \cite[Corollary 5.2]{MR986363} we have 
\[
\e^{\kappa}\ge\frac{1}{\left(1-\lambda_{0}\right)^{2}}.
\]
We obtain 
\[
\e^{\delta}\ge n\e^{-\kappa/2}+\sqrt{n^{2}\e^{-\kappa}-(2n-1)}.
\]
Multiplying by $\e^{\kappa/2}$ yields 
\begin{equation}
\e^{\delta+\kappa/2}\ge n+\sqrt{n^{2}-(2n-1)\e^{\kappa}}.\label{eq:estimate 1}
\end{equation}
A short calculation shows that 
\begin{align*}
\sqrt{n^{2}-(2n-1)\e^{\kappa}} & =\sqrt{n^{2}-(2n-1)-(2n-1)(\e^{\kappa}-1)}\\
 & \ge\sqrt{n^{2}-(2n-1)}-\sqrt{(2n-1)(\e^{\kappa}-1)}\\
 & \ge n-1-\left(\e^{\delta(F_{n})}\e^{\kappa}\right)^{1/2}.
\end{align*}
Combining with (\ref{eq:estimate 1}) we see that 
\begin{align*}
\e^{\delta+\kappa/2} & \ge2n-1-\e^{\delta(F_{n})/2}\e^{\kappa/2}=\e^{\delta(F_{n})}-\e^{\delta(F_{n})/2}\e^{\kappa/2},
\end{align*}
which implies 
\[
\left(\e^{\delta}+\e^{\delta(F_{n})/2}\right)\e^{\kappa/2}\ge\e^{\delta(F_{n})}.
\]
Finally, since $\delta>\frac{1}{2}\delta(F_{n})$ we deduce that 
\[
2\e^{\delta}\e^{\kappa/2}>\e^{\delta(F_{n})},
\]
which finishes the proof.
\end{proof}

\def\cprime{$'$}
\providecommand{\bysame}{\leavevmode\hbox to3em{\hrulefill}\thinspace}
\providecommand{\MR}{\relax\ifhmode\unskip\space\fi MR }
\providecommand{\MRhref}[2]{%
  \href{http://www.ams.org/mathscinet-getitem?mr=#1}{#2}
}
\providecommand{\href}[2]{#2}

\end{document}